\documentclass[a4paper,reqno,11pt]{amsart}

\usepackage{titlesec}
\usepackage{amsmath,amsthm,amsfonts}
\usepackage{fullpage}
\usepackage{t1enc}
\usepackage{enumerate}
\usepackage{epsfig}
\usepackage{amssymb}
\usepackage{graphicx}
\usepackage{listings}
\usepackage{gensymb}
\usepackage{tikz}
\usepackage{tikz-cd}
\usepackage[utf8]{inputenc}
\usepackage[utf8]{luainputenc}
\usepackage[english]{babel}
\usepackage[top=1.5in,bottom=1.2in,left=1in,right=1in]{geometry}
\usepackage{mathtools}
\usepackage{pgf}
\usepackage[colorlinks=true, linkcolor=teal, citecolor=blue]{hyperref}
\usepackage{cleveref}
\usepackage{faktor}
\usepackage{xcolor}
\usepackage{xfrac}
\usepackage[all,cmtip]{xy}
\usepackage{xymtex}
\usepackage{xymtx-pdf}
\usepackage{longtable}
\usepackage{hyperref}

\titleformat{\section}
  {\centering\normalfont\scshape\fontsize{12}{17}\bfseries}
  {\thesection}
  {1em}
  {}
\titleformat{\subsection}
  {\normalfont\fontsize{12}{14}\bfseries}
  {\thesubsection}
  {1em}
  {}

\newcommand{\Cp}[1]{\mathbb{C}\mathrm{P}^{#1}}

\newcommand{\N}{\mathbb{N}}
\newcommand{\Z}{\mathbb{Z}}

\newcommand{\R}{\mathbb{R}}

\newcommand{\Sp}[1]{\mathbb{S}^{#1}}

\newcommand{\fbf}[1]{[{#1},\operatorname{\mathit{BSF}}]}

\newcommand{\wi}[1]{\widetilde{#1}}

\newcommand{\cc}[1]{\mathcal{C}{\left(#1\right)}}
 \newcommand{\csum}{{\#}}
\newcommand{\ra}{\rightarrow}

\newcommand{\lr}{\longrightarrow}

\newcommand{\opl}[1]{\oplus{_{#1}}}
\newcommand{\kr}[1]{\mathrm{Ker}{\left(#1\right)}}

\newcommand{\im}[1]{\mathrm{Im}{\left(#1\right)}}

\newcommand{\pii}[1]{\pi_{#1}}
\newcommand{\ol}[1]{\overline{#1}}

\newcommand{\bpi}{\scalebox{1}[1.2]{$\pi$}}
\newcommand{\bmu}{\scalebox{1}[1.2]{$\mu$}}
\newtheoremstyle{defis}%
    {3pt}% Space above
    {3pt}% Space below
    {\normalshape}% Body font
    {}% Indent amount
    {\normalshape}% Head font
    {.}% Punctuation after theorem head
    {.5em}% Space after theorem head
    {}% Theorem head spec (can be left empty, meaning 'normal')

\newtheorem{thm}{Theorem}[section]
\newtheorem{lemm}[thm]{Lemma}

\newtheorem{thrmnonum}{Theorem}

\theoremstyle{defis}

\newtheorem{rem}[thm]{Remark}

\numberwithin{equation}{section}

\Crefname{lemm}{Lemma}{Lemmas}
\crefname{lemm}{lemma}{lemmas}
\crefname{thm}{theorem}{theorems}
\Crefname{thm}{Theorem}{Theorems}

\newcounter{casenum}

\newcounter{casenum1}

\title{Tangential smoothings of $k$-fold connected sum of complex projective spaces}

\numberwithin{equation}{section}
\begin{document}

\author{Priyanka Magar-Sawant}
	
	\address{Department of Mathematics,
		Indian Institute Of Technology Bombay, India-400076}
\email{priyanka.ms.math@gmail.com}

% 	\subjclass [2020] {Primary : {55Q10, 55R50; Secondary : 55Q55, 55N15}}
	
%  \keywords{Stable cohomotopy groups, K and KO-groups, k-fold connected sum, complex projective spaces.}

%%%%%%%%%%%%%%%%%%%%%%%%%%%%%%%%%%%%%%%%%%%%%%%%%%%%%%%%%%%%%%%%%%%%%%%%%%%%%%%%%%%%%%%%%%%%%%%%%%%%%%%%%%%%%%%%%%%%%%%%%%%%%%%%%%%%%%%%%%%%%%%%%%%%%%%%%%%%%%%%%%%%%%%%%%%%%%%%%%%%%%%%%%%%%%%%%%%%%%%%%%%%%%%%%%%%%%%%
\maketitle
\begin{abstract}
This paper determines tangential structure set of $\#_k\Cp{n}$, for $3\leq n \leq 7$, by analyzing their stable cohomotopy groups and $KO$-groups. As a consequence, it establishes the existence of manifolds with tangential homotopy type but not homeomorphic to $\#_k\Cp{n}$ for $n=4,5$.
\end{abstract}

%%%%%%%%%%%%%%%%%%%%%%%%%%%%%%%%%%%%%%%%%%%%%%%%%%%%%%%%%%%%%%%%%%%%%%%%%%%%%%%%%%%%

%%%%%%%%%%%%%%%%%%%%%%%%%%%%%%%%%%%%%%%%%%%%%%%%%%%%%%%%%%%%%%%%%%%%%%%%%%%%%%%%%%%%%%%%%%%%%%%%%%%%%%%%%%%%%%%%%%%%%%%%%%%%%
%
\section{Introduction}

%%%%%%%%%%%%%%%%%%%%%%%%%
In surgery theory, the tangential structure sets for a given manifold $M$, play a crucial role \cite{CodimensionTwoSouls,crowleyfinitegroupaction2015,AmbientSurgery,HomologySpheresasStationarysets}. In this paper, we compute these sets for the $k$-fold connected sum of complex projective spaces. The method uses the surgery exact sequence, which has remarkable connections with stable cohomotopy groups, prompting us to proceed with these computations.
 % Here,  $SF$ is the limit of the spaces $SF_n$ over $n$, where  $SF_n$ consist of a base point preserving degree one maps from $\Sp{n-1}$ to itself.
Several computations of the stable cohomotopy groups of spheres and projective spaces have been done by Brumfiel, Toda, Ravenal, R. West, and others \cite{brumfiel1968,Brumfiel1971HomotopyEOUnpublished,TodaBook,Ravenel,WESTSomeCohomotopyofProjectiveSpace,RKSmoothStructuresonComplexProjectiveSpaces}. 
In this paper, computations of the reduced $0^\text{th}$ stable cohomotopy group of $\#_k\Cp{n}$, denoted by $\pi_s^0(\#_k\Cp{n})$, for $3\leq n\leq 8$, are carried out. (see \Cref{co-homotopy groups of cpn,pi0 of cp4 and 7})

It is important to highlight that, for a manifold $X$, $\pi_s^0(X)$ can be identified with the set of homotopy classes of maps from $X$ to the space $SF$, denoted by $[X,SF]$. Note that, the space $SF $ is a part of the fibration $SF\lr F/O\lr BSO$, where $F/O$ denotes the homotopy fiber of the map $BSO\lr BSF$. Here, $BSO$ and $BSF$ are the classifying spaces for stable real vector bundles and stable spherical fibrations, respectively, \cite[Lemma 10.6]{Wall}, \cite[Sections 2 and 3]{MilgramTheClassifyingSpacesForSurgeryAndCobordismOfManifolds}. Brumfile's result \cite[pg 401]{brumfiel1968}, 
is focused on the splitting of $[\Cp{n},F/O]$. \Cref{thm: f/o splitting} is a generalization to the Brumfile's findings on $k$-fold connected sum of complex projective spaces. Importantly, it is observed that the torsion part of this group is isomorphic to $\pi_s^0(\#_k\Cp{n})$, for all $k,n\in\N$. The following result summarizes our findings, together with Brumfile's result at $k=1$;
\begin{thrmnonum}\label{f/o spitting in intro}
     Let $k,n\in\N$. Then
        $$[\#_k\Cp{n},F/O]\cong\begin{cases}
            \pi_s^0(\#_k\Cp{n})\opl{}\Z^{k[\frac{n}{2}]} & \text{  if } n \text{ odd }\\
            \pi_s^0(\#_k\Cp{n})\opl{}\Z^{k[\frac{n-1}{2}]+1} & \text{  if } n \text{ even }.
         \end{cases} $$
\end{thrmnonum}
Here, it is worth mentioning that the free part corresponds to the free part of the group $[\#_k\Cp{n},BSO]$. Additionally, note that the group $[\#_k\Cp{n},BSO]$ coincides with the reduced real $K$-groups of $\#_k\Cp{n}$, denoted by $\wi{KO}^0(\#_k\Cp{n})$. Thus computations of these groups are conducted in the paper.

\begin{thrmnonum}\label{thm B}
    The groups $\wi{KO}^i(\#_k\Cp{n})$ for $-7\leq i\leq 0$ and $k,n\in\N$ with $k\geq 2$ are determined along with their generators. (see \Cref{thm: ko group})
\end{thrmnonum}

%------------------------------------------------

Above results lead to the comparison of tangential homotopy smoothings and $PL$-tangential smoothings of $\#_k\Cp{n}$ , as outlined below:

\begin{thrmnonum}\label{thm C}\ Let $M$ be a closed smooth $2n$-manifold, and $k\in\N$. 
\begin{enumerate} 
    \item For any $k,n\in\N$, every homeomorphism from $M$ to $\#_k\Cp{n}$ is a tangentially homotopy equivalence.

    \item  For $n=3,6$ and $7$, every tangentially homotopy equivalence from $M$ to $\#_k\Cp{n}$ is a homeomorphism.
    
    \item  There exist exactly $2^k$  smooth manifolds that are tangentially homotopy equivalent to $\#_k\Cp{4}$ but not homeomorphic to it.

    \item There exist exactly $2^{k-2}$ smooth manifolds that are tangentially homotopy equivalent to $\#_k\Cp{5}$ but not homeomorphic to it.
    \end{enumerate}
\end{thrmnonum}
\Cref{thm C} is conclusion of \Cref{thm: stpl to stdiff inje,thm: connetced sum Cpn PL isomorphic to PL/o,thm: stdiff}.

\vspace{5mm}
%%%%%%%%%
\noindent {Organization of the paper:} In Section \ref{sec 2}, we compute the groups $[\#_k\Cp{n},SF]$ based on the known computations of $[\Cp{n},SF]$ for $4 \leq n\leq 8$. Moving to Section \ref{sec: K groups}, we determine the groups $[\#_k\Cp{n},BSO]$ and their generators for all values of $k,n \in \N$. Section \ref{sec: normal invariant set} establishes that, the smooth normal invariant set $\mathcal{N}(\#_k\Cp{n})$ is the direct sum of $[\#_k\Cp{n},SF]$ and the free part of $[\#_k\Cp{n},BSO]$, for all $k,n\in\N$.
Finally, in Section \ref{smooth tangential structure set}, we explore additional results related to the classification problem, ultimately establishing \Cref{thm C}.

%%%%%%%%%%%%%%%%%%%%%%%%%%%%%%%%%%%%%%%%%%%%%%%%%%%%%%%%%%%%%%%%%%%%%%%%%%%%%%%%%%%%%%%%%%%%%%%%%%%%%%%%%%%%%%%%%%%%%%%%%%%%%
\section{Stable cohomotopy groups}\label{sec 2}

 In order to derive the splitting in the Theorem \ref{f/o spitting in intro} and computations of \Cref{thm B}, we use the fiber sequence 
\begin{equation}\label{f/o fiber seq}
    \begin{tikzcd}[sep=small]\cdots & SO &  SF & {F/O} & BSO & BSF & \cdots
	\arrow[from=1-1, to=1-2]
	\arrow["a", from=1-2, to=1-3]
	\arrow["b", from=1-3, to=1-4]
	\arrow["c", from=1-4, to=1-5]
	\arrow["f", from=1-5, to=1-6]
	\arrow[from=1-6, to=1-7].
 \end{tikzcd}
\end{equation}  
Note that, all spaces in this fiber sequence are infinite loop spaces, hence the resulting exact sequence consists of all abelian groups.

%%%%%%%%%%%%%%%%%%%%%%%%%%%%%%%%%%%%
Let us start with calculating the reduced $0^\text{th}$ stable cohomotopy groups $\pi_s^0(\#_k\Cp{n})$ for $3\le n \le 8$ and $k\geq 2$. These computations rely on pre-existing results and computations that are presented below.  

In \cite{brumfiel1968}, Brumfiel focused on the computations of $0^\text{th}$ stable cohomotopy groups of individual copies of the complex projective spaces in lower dimensions; the groups are as follows:

\begin{thm}[\cite{brumfiel1968}]{\mbox{}}
  \begin{enumerate}
        \item $\pi_s^0(\Cp{3}) \cong \Z_2$
    \item $\pi_s^0(\Cp{4}) \cong \Z_2^2$
    \item $\pi_s^0(\Cp{5}) \cong \Z_2^2\opl{}\Z_3$
  \item   $\pi_s^0(\Cp{6}) \cong \Z_2\opl{}\Z_3$
    \end{enumerate}
\end{thm}

Furthermore, R. Kasilingam \cite{RKSmoothStructuresonComplexProjectiveSpaces} computed the following results:
\begin{thm}[\cite{RKSmoothStructuresonComplexProjectiveSpaces}]{\mbox{}}
% The $0^\text{th}$ stable cohomotopy groups of $\Cp{n}$ for $n=7$ and $8$ are
  \begin{enumerate}
        \item   $\pi_s^0(\Cp{7}) \cong \Z_2^3$
\item   $\pi_s^0(\Cp{8}) \cong \Z_2^3$
    \end{enumerate}
\end{thm}

%%%%%%%%%%%%%%%%%%%%%%%%%%%%%%%%%%%%%
Our computations mainly utilize the following short exact sequence.
\begin{thm}[{\cite{SB-RK-Priya2023smooth},Theorems 4.1 and 4.2}]
    Let $ M $ be a closed oriented smooth $ n $-manifold. Then there exists a short exact sequence
\begin{equation}\label{seq: connected sum Y short exact seq}
    \begin{tikzcd}[ampersand replacement=
 \&, sep=small]
	0 \& {\faktor{\pii{n}(Y)}{\im{(\Sigma h)^*}}} \& {[\csum_{i\in I}M,Y]} \& {\big(\opl{k-1}[M^{(n-1)},Y]\big)\oplus\kr{h^*}} \& 0,
	\arrow[from=1-1, to=1-2]
	\arrow["{\wi{d^*}}", from=1-2, to=1-3]
	\arrow["{\iota^*}", from=1-3, to=1-4]
	\arrow[from=1-4, to=1-5]
\end{tikzcd}
\end{equation}
 where $ \xi=(\vee_kh)\circ p $ is the attaching map for $ \csum_kM $ with $ h:\Sp{n-1}\to M^{(n-1)} $ as the attaching map for $ M $ and $ p $ is the pinch map.
\end{thm}

In \eqref{seq: connected sum Y short exact seq}, consider $M_i=\Cp{n}$ for all $i$, and substituting $Y$ to be the homotopy-commutative $H$-group $SF$. Recall that, $SF$ is the limit of the spaces $SF_n$ over $n$, where  $SF_n$ consist of base point preserving degree one maps from $\Sp{n-1}$ to itself. Therefore we have the following short exact sequence, 
\begin{equation}\label{short exact sequence for stable}
    \begin{tikzcd}[sep=small]
	0 & {\faktor{\pii{2n}^s}{\im{(\Sigma h)^*}}} & {\pii{s}^0(\#_k\Cp{n})} & {\underset{k-1}{\oplus}\pii{s}^0(\Cp{n-1})\opl{}\kr{h^*}} & 0.
	\arrow[from=1-1, to=1-2]
	\arrow["{\wi{d}^*}", from=1-2, to=1-3]
	\arrow["{\iota^*}", from=1-3, to=1-4]
	\arrow[from=1-4, to=1-5]
\end{tikzcd}
\end{equation}
The next theorem uses the computations done for the groups $ \im{(\Sigma h)^*} $
and $ \kr{h^*} $ in \cite{brumfiel1968} and \cite[Proposition 2.8, Theorem 2.9]{RK-2017-CP5-8} for $3\leq n\leq 8$.

%%%%%%%%%%%%%%%%%%%%

\begin{thm}\label{co-homotopy groups of cpn}
\begin{enumerate}[(i)]

\item There is an isomorphism
   $$d^*:\pi_6^s\lr \pii{s}^0(\#_k\Cp{3}),$$
where $ \pi_6^s=\Z_2 $.

\item  There is a short exact sequence
		\[\begin{tikzcd}
			0 & {\faktor{\pi_8^s}{\Z_2}} & {\pii{s}^0(\#_k\Cp{4})} & {\underset{k}{\oplus}\pii{s}^0(\Cp{3})} & 0,
			\arrow[from=1-1, to=1-2]
			\arrow["{\widetilde{d^*}}", from=1-2, to=1-3]
			\arrow["{{\iota^*}}", from=1-3, to=1-4]
			\arrow[from=1-4, to=1-5]
		\end{tikzcd}\]
		where $\pi_8^s=\Z_2^2$  and $\pii{s}^0(\Cp{3})\cong\Z_2 $. %\tcc{$ \im{\gamma^*}=\im{h^*}=0 $ and $ \im{\Sigma \gamma^*}=\Z_2 $}

\item  There is a short exact sequence
		\[\begin{tikzcd}
			0 & {\pi_{10}^s} & {\pii{s}^0(\#_k\Cp{5})} & {\underset{k-1}{\oplus}\pii{s}^0(\Cp{4})\opl{}\kr{h^*}} & 0,
			\arrow[from=1-1, to=1-2]
			\arrow["{{d^*}}", from=1-2, to=1-3]
			\arrow["{{\iota^*}}", from=1-3, to=1-4]
			\arrow[from=1-4, to=1-5]
		\end{tikzcd}\]
		where $ \pi_{10}^s=\Z_2\opl{}\Z_3 $, $ \pii{s}^0(\Cp{4})\cong\Z_2^2 $ and $ \kr{h^*:\pii{s}^0(\Cp{4})\ra\pi_{9}^s}\cong\Z_2 $.

\item There is an isomorphism
   $$\iota^*:\pii{s}^0(\#_k\Cp{6})\ra\underset{k-1}{\oplus}\pii{s}^0(\Cp{5})\oplus\kr{h^*},$$
where $ \pii{s}^0(\Cp{5})\cong\Z_2^2\oplus\Z_3 $ and $ \kr{h^*:\pii{s}^0(\Cp{5})\ra\pii{11}^s} \cong\Z_2\opl{}\Z_3$.

\item  There is a short exact sequence 
		\[\begin{tikzcd}
			0 & {\pi_{14}^s} & {\pii{s}^0(\#_k\Cp{7})} & {\underset{k-1}{\oplus}\pii{s}^0(\Cp{6})\oplus\kr{h^*}} & 0,
			\arrow[from=1-1, to=1-2]
			\arrow["{{d^*}}", from=1-2, to=1-3]
			\arrow["{{\iota^*}}", from=1-3, to=1-4]
			\arrow[from=1-4, to=1-5]
		\end{tikzcd}\] 
		where  $ \pi_{14}^s=\Z_2^2 $, $ \pii{s}^0(\Cp{6})\cong\Z_2\oplus\Z_3 $ and $ \kr{h^*:\pii{s}^0(\Cp{6})\ra\pi_{13}^s}\cong\Z_2 $.

\item  There is a short exact sequence
		\[\begin{tikzcd}
			0 & {\faktor{\pii{16}^s}{\Z_2}} & {\pii{s}^0(\#_k\Cp{8})} & {{\underset{k-1}{\oplus}}\pii{s}^0(\Cp{7})\oplus\kr{h^*}} & 0,
			\arrow[from=1-1, to=1-2]
			\arrow["\wi{d^*}", from=1-2, to=1-3]
			\arrow["{{\iota^*}}", from=1-3, to=1-4]
			\arrow[from=1-4, to=1-5]
		\end{tikzcd}\]  
		where $\pii{16}^s=\Z_2^2$, $ \pii{s}^0(\Cp{7})\cong\Z_2^3 $ and $ \kr{h^*:\pii{s}^0(\Cp{7})\ra\pi_{15}^s}\cong\Z_2^2 $.

\end{enumerate}
\end{thm}
\vspace{1em}
%%%%%%%%%%%%%%%%%%%%%%%%%%%%%%%%%%%%%%%%%

From \Cref{co-homotopy groups of cpn}, we obtain that $\pii{s}^0(\#_k\Cp{3})\cong\Z_2$ and $\pii{s}^0(\#_k\Cp{6}) \cong\Z_2^{2k-1}\opl{}\Z_3^k$. Now we compute the groups $ \pii{s}^0(\#_k\Cp{4}) $ and $ \pii{s}^0(\#_k\Cp{7}) $ by proving that the aforementioned short exact sequences $(ii)$ and $(v)$ split.
The group $\pi^0_s(\#_k\Cp{5})$ is computed in the last section of the paper, as the proof requires additional results discussed therein.\\

 We begin by constructing a cofiber sequence:  For $ i\in I =\{1,2,\dots,k\}$, let $ M_i $ be a closed oriented smooth $ n $-manifold with the attaching map $ h_i:\Sp{n-1}_i\lr M_i^{(n-1)} $. Consider, for $  I'=\{1,2,\cdots,k-1\} $, the map $ \bmu:\sqcup_{I'}\Sp{n-1}_i\hookrightarrow\#_{i\in I}M_i $ that maps each copy of $ \Sp{n-1}_i $ on the gluing part of $ M_i $ while taking the connected sum. Observe that the homotopy cofiber of the map $ \bmu $ is $ \vee_{i\in I} M_i $. 

Now, consider the map $ \lambda$ which is the composition map $j\circ \vee_{i\in I}~p_i \circ \sqcup_{i\in I'}\lambda_i$ as shown in the following homotopy commutative diagram,
\begin{equation}
\begin{tikzcd}\label{eta map}
	{\sqcup_{i\in I'}\Sp{n-1}_i} & {\#_{i\in I}M_i} \\
	{\vee_{i\in I}\widehat{M_i}} & {\vee_{i\in I}M_i^{(n-1)},}
	\arrow["\lambda", from=1-1, to=1-2]
	\arrow["{\underset{i\in I'}{\sqcup}\lambda_i}"', hook, from=1-1, to=2-1]
	\arrow["\underset{i \in I}{\vee}p_i"', from=2-1, to=2-2]
	\arrow["j"', hook, from=2-2, to=1-2]
\end{tikzcd}
\end{equation}
where $ \widehat{M_i} $ is obtained by deleting a smaller open disc from the top cell of $M_i $ while taking the connected sum, and  $\lambda_i:\Sp{n-1}_i\hookrightarrow \widehat{M_i} $ is the inclusion on that boundary part. Note that, the important reason of constructing $\lambda$ in such a way is that for all $i \in I$, the map $ p_i $ is a homotopy equivalence, and $ p_i\circ \lambda_i$ is homotopic to the attaching map $h_i$.

Consider the following diagram obtained from cofiber sequences associated with maps $ \bmu $ and $\lambda $ 
\begin{equation}\label{seq3}
\begin{tikzcd}
	{\sqcup_{i\in I'}\Sp{n-1}_i} & {\#_{i\in I}M_i} & {\vee_{i\in I}M_i} & \dots \\
	{\sqcup_{i\in I'}\Sp{n-1}_i} & {\#_{i\in I}M_i} & {C_{\lambda}} & \dots
	\arrow["\bmu", hook, from=1-1, to=1-2]
	\arrow["\rho", from=1-2, to=1-3]
	\arrow["\mathrm{Id}", from=1-1, to=2-1]
	\arrow["\lambda"', from=2-1, to=2-2]
	\arrow["q"', from=2-2, to=2-3]
	\arrow["\beta", dashed, from=1-3, to=2-3]
	\arrow[from=1-3, to=1-4]
	\arrow[from=2-3, to=2-4]
	\arrow["\mathrm{Id}", from=1-2, to=2-2]
\end{tikzcd}
\end{equation}
in which the leftmost square commutes up to homotopy. This results in the existence of the map $\beta:\vee_{i\in I}M_i\lr C_{\lambda}$ such that the second square commutes up to homotopy, and consequently, all squares commute up to homotopy. Applying the contra-variant functor $[\_,Y]$ on the above sequences yields $ [C_{\lambda},Y] \cong\opl{i \in I}[M_i,Y]$, for an homotopy associative $ H $-group $ Y $. \\
%%%%%%%%%%%%%

In the forthcoming proof of the theorem, the cofiber sequence resulting from the constructed map $\lambda$ will play a vital role.
%%%%%%%%%%%%%%%%%%%%%%%%%%%%%%%%%%%%%%%%%%%%%%%%%%%%

\begin{thm}\label{pi0 of cp4 and 7}
Let $ k\in \N $. Then
\begin{enumerate}[(i)]
\item $ \pii{s}^0(\#_k\Cp{4})\cong \Z_2^{k+1}$

% \item $\pii{s}^0(\#_k\Cp{5}) \cong\Z_2^{2k}\opl{}\Z_3$

\item $ \pii{s}^0(\#_k\Cp{7})\cong \Z_2^{k+2}\opl{}\Z_3^{k-1} $
\end{enumerate}
\end{thm}
\begin{proof}
\begin{enumerate}[$(i)$]
    \item Consider the short exact sequence from the \Cref{co-homotopy groups of cpn} $ (ii) $,
\begin{equation*}
        \begin{tikzcd}
			0 & {\faktor{\pi_8^s}{\Z_2}} & {\pii{s}^0(\#_k\Cp{4})} & {\underset{k}{\oplus}\pii{s}^0(\Cp{3})} & 0,
			\arrow[from=1-1, to=1-2]
			\arrow["{\widetilde{d^*}}", from=1-2, to=1-3]
			\arrow["{{\iota^*}}", from=1-3, to=1-4]
			\arrow[from=1-4, to=1-5]
		\end{tikzcd}
\end{equation*}
 The splitting of this sequence is equivalent to the non-existence of order $4$ element in the group $ \pii{s}^0(\#_k\Cp{4}) $. Consider the map $ \lambda:\sqcup_{k-1}\Sp{7}\ra\#_k\Cp{4} $ from \eqref{eta map} and the cofiber sequence associated to $ \lambda $ from \eqref{seq3}. We know that the map $\lambda$ factors though $(\vee_{k-1}h)$ up to homotopy, where $h:\Sp{7}\lr \Cp{3}$ is the Hopf map. In \cite{brumfiel1968}, it is proved that the induced map $h^*:\pii{s}^0(\Cp{3})\lr\pii{7}^s$ is trivial, and consequently $ \lambda^* $ as well. Thus, in the long exact sequence obtained from $\lambda$ by applying functor $ [\_,SF] $, the induced map $q^*:\opl{k}\pii{s}^0(\Cp{4})\lr \pii{s}^0(\#_k\Cp{4})$ is surjective. Finally, the group $\pii{s}^0(\Cp{4})\cong\Z_2^2$ concludes that the $\pii{s}^0(\#_k\Cp{4})$ has no order $4$ element.

% \item 

\item  Now consider the short exact sequence from \Cref{co-homotopy groups of cpn} $ (v) $
\begin{equation*}\begin{tikzcd}
			0 & {\pi_{14}^s} & {\pii{s}^0(\#_k\Cp{7})} & {\underset{k-1}{\oplus}\pii{s}^0(\Cp{6})\oplus\kr{h^*}} & 0,
			\arrow[from=1-1, to=1-2]
			\arrow["{{d^*}}", from=1-2, to=1-3]
			\arrow["{{\iota^*}}", from=1-3, to=1-4]
			\arrow[from=1-4, to=1-5]
		\end{tikzcd}
\end{equation*}
The primes occurring in the primary decomposition of the group $\pii{s}^0(\#_k\Cp{7})$ are $2$ and $3$. Note that the localization of the short exact sequence at $3$, results in  $\pii{s}^0({\#_k\Cp{7}})_{(3)}\cong\Z_3^{k-1}$. Thus, the sequence splits if the group $ \pii{s}^0(\#_k\Cp{7}) $ has no order $4$ element.
Consider the map $ \lambda:\sqcup_{k-1}\Sp{13}\lr \#_k\Cp{7} $ from \eqref{eta map} and the cofiber sequence associated to $ \lambda $ from \eqref{seq3}. The induced map $\lambda^*:\pii{s}^0(\#\Cp{7})_{(2)}\lr\pii{13_{(2)}}^0$, in the long exact sequence of cohomotopy groups, is trivial as the group $\pii{13_{(2)}}^s=0$. This implies the map $q^*:\opl{k}\pii{s}^0(\Cp{7})_{(2)}\lr\pii{s}^0(\#_k\Cp{7})_{(2)} $ is surjective. Since the group $\pii{s}^0(\Cp{7})\cong\Z_2^3$ implies $ \pii{s}^0(\#_k\Cp{7}) $ has no order $ 4 $ element.

\end{enumerate}
This completes the proof for both cases.
\end{proof}

%%%%%%%%%%%%%%%%%%%%%%%%%%%%%%%%%%%%%%%%%%%%%%%%%%%%%%%%%%%%%%%%%%%%%%%%%%%%%%%%%%%%

%%%%%%%%%%%%%%%%%%%%%%%%%%%%%%%%%%%%%%%%%%%%%%%%%%%%%%%%%%%%%%%%%%%%%%%%%%%%%%
\section{Reduced K-groups}\label{sec: K groups}

Now, let us direct our attention to the groups $[\#_k\Cp{n},SO]$ and $[\#_k\Cp{n},BSO]$, which play a vital role in the fiber sequence \eqref{f/o fiber seq}. Recall that, the space $BSO$ serves as the classifying space for the stable real vector bundle, where $SO$ represents the fiber of the universal bundle $ESO\lr BSO$.  Also there exists an identification between the group $[\#_k\Cp{n},BSO]$ and the reduced real $K$-group of $\#_k\Cp{n}$.

In fact, in this section, we compute the reduced real and complex $K$-groups of $\#_k\Cp{n}$, for all $k,n\geq 2$, along with the generators. At $k=1$, these groups are studied in \cite{FujjiKO-grOfCP} for all $n\in\N$.

Let us fix the notations $\wi{K}^s(X)$ for the reduced complex $s^\text{th}$ $K$-group of a space $X$, and $\wi{KO}^s(X)$ for the reduced real $s^\text{th}$ $K$-group, where $s\in \Z$. 

For the forthcoming computations, we often use the following short exact sequence, obtained from \eqref{seq: connected sum Y short exact seq} by considering $M_i=\Cp{n}$ for all $i\in I$,
\begin{equation}\label{seq 1 general short exact sequence}
    \begin{tikzcd}[sep=small]
	0 & {\faktor{\pi_{2n+s}(Y)}{\im{\Sigma^{s+1}h}^*}} & {[\Sigma^s\#_k\Cp{n},Y]} & {\underset{k-1}{\oplus}[\Sigma^s\Cp{n-1},Y]\opl{}\kr{\Sigma^s h}^*} & 0~,
	\arrow[from=1-1, to=1-2]
	\arrow["{\wi{d^*}}", from=1-2, to=1-3]
	\arrow["{\iota^*}", from=1-3, to=1-4]
	\arrow[from=1-4, to=1-5]
\end{tikzcd}
\end{equation}
where $d:\Sigma^s\#_k\Cp{n}\lr \Sp{2n+s}$ is the degree one map, $\iota:\vee_{k}\Sigma^s\Cp{n-1}\lr \Sigma^s\#_k\Cp{n}$ be an inclusion map, and $h:\Sp{2n-1}\lr \Cp{n-1}$ be the Hopf map.

% Let's begin with the computations of the groups $\wi{K}^s(\#_k\Cp{n})$.

%%%%%%%%%%%%%%%%%%%%%%%%%%%%%%%%%%%%%%%%%%%%%%%%%%%%%%%%%%%%%%%%%%%%%%%%%%%%%%%%%%%%
\subsection{Complex K-groups}

Let us consider the following short exact sequence which is obtained by substituting $Y=BSU$ in the sequence \eqref{seq 1 general short exact sequence}, along with the identification $[\Sigma^sX,BSU]=\wi{K}^{-s}(X)$ for $s\in \N$,
\begin{equation}\label{seq 1 for k-groups}
    \begin{tikzcd}[sep=small]
	0 & {\faktor{\wi{K}^{-s}(\Sp{2n})}{\im{\Sigma^{s+1}h}^*}} & {\wi{K}^{-s}(\#_k\Cp{n})} & {\underset{k-1}{\opl{}}\wi{K}^{-s}(\Cp{n-1})\opl{}\kr{\Sigma^s h}^*} & 0.
	\arrow[from=1-1, to=1-2]
	\arrow["{\wi{d^*}}", from=1-2, to=1-3]
	\arrow["{\iota^*}", from=1-3, to=1-4]
	\arrow[from=1-4, to=1-5]
\end{tikzcd}
\end{equation}

Due to Bott periodicity, it is enough to compute the groups $\wi{K}^0(\#_k\Cp{n})$ and $\wi{K}^{-1}(\#_k\Cp{n})$ for all $k,n\geq 2$. 

Let $\omega$ be a generator of the free abelian group $\wi{K}^0(\Sp{2n})$. Let $H$ be the canonical complex line bundle over $\Cp{n}$, and let $\eta:=H-1$. Then the set $\{ \eta^j~|~j=1,2,\cdots,n\}$ forms an integral basis for the free abelian group $\wi{K}^0(\Cp{n})$ for all $n\in \N$ \cite{FujjiKO-grOfCP}.

Let us consider the case where $s=0$ in the sequence \eqref{seq 1 for k-groups}. In this situation, we have $\wi{K}^{-1}(\Cp{n-1})=0$, which implies that the map $\Sigma h^*:\wi{K}^{-1}(\Cp{n-1})\lr \wi{K}^0(\Sp{2n})$ becomes trivial. Furthermore, the group $\wi{K}^0(\Sp{2n-1})$ is also zero, leading to $\kr{h}^*=\wi{K}^{0}(\Cp{n-1})$. Since both $\wi{K}^0(\Sp{2n})$ and $\wi{K}^{0}(\Cp{n-1})$ are free abelian groups, we obtain a split short exact sequence, which allows us to make the following statement.

% \begin{thm}
%     The following set form an integral basis for the group $\wi{K}^0(\#_k\Cp{n})$
%     $$\{ d^*(\omega), \eta_i^j~|~i=1,2,\cdots,k, ~j=1,2,\cdots,n-1 \}$$
%     for all $k,n\in\N$.
% \end{thm}
% or
\begin{thm}
    The $\wi{K}^0(\#_k\Cp{n})$ is a free module with a basis 
    $$\{ d^*(\omega), \eta_i^j~|~i=1,2,\cdots,k, ~j=1,2,\cdots,n-1 \}$$
    for $k,n\in\N$.
\end{thm}

\noindent\textbf{Note}: In this section, we are using the same notation $\eta^j_i$ instead of $\gamma^*(\eta_i^j)$ if $\gamma^*$ denotes the left inverse of $\iota^*$ associated to the split exact sequence.

Next, let us consider $s=1$ in the sequence \eqref{seq 1 for k-groups}. Then since $\wi{K}^{-1}(\Sp{2n})$ and $\wi{K}^{-1}(\Cp{n-1})$ are both trivial, for all $n\in\N$, we get the following result.
\begin{thm}
     The group $\wi{K}^{-1}(\#_k\Cp{n})$ is trivial for all $k,n\in\N$.
\end{thm}

With these computations, we have determined the complex $K$-groups of $\#_k\Cp{n}$ for all values of $k,n\geq 2$.

%%%%%%%%%%%%%%%%%%%%%%%%%%%%%%%%%%%%%%%%%%%%%%%%%%%%%%%%%%%%%%%%%%%%%%%%%%%%%%%%%%%%%%%%%%%%%%%%
%%%%%%%%%%%%%%%%%%%%%%%%%%%%%%%%%%%%%%%%%%%%%%%%%%%%%%%%%%%%%%%%%%%%%%%%%%%%%%%%%%%%%%%%%%%%%%%%
\subsection{Real K-groups}
Now let us proceed with the computations of the real $K$-groups of $\#_k\Cp{n}$. Due to Bott periodicity it is enough to compute $\wi{KO}^{m}(\#_k\Cp{n})$ within the range $-7\leq m\leq 0$. To accomplish this, we will primarily rely on the exact sequence introduced in the proof of the following lemma.
\begin{lemm}\label{lemm: skeleton sandwitch sesquence}
Let $M_i$ be an $n$-manifold with $(n-1)$-skeleton $M^{(n-1)}_i$ for $i\in I=\{1,2,\cdots,k\}$, and the space $Y$ be a homotopy-commutative $H$-group. If the groups $[M_k,Y]$, $[M^{(n-1)}_i,Y]$ are trivial for all $i\in I'=\{1,2,\cdots,k-1\}$ then the group $[\#_{i\in I}M_i,Y]$ is also trivial. 
\end{lemm}
\begin{proof}
 The proof follows directly by considering the following induced long exact sequence 
\begin{equation}\label{seq 2 (n-1 coppies to connected sum to single copy)}
    \begin{tikzcd}
	\cdots & {[\underset{i\in I'}{\vee}\Sigma M^{(n-1)}_i,Y]} & {[M_k,Y]} & {[\underset{i\in I}{\#}M_i,Y]} & {[\underset{i\in I'}{\vee}M^{(n-1)}_i,Y]\cdots}
	\arrow["{\iota^*}", from=1-4, to=1-5]
	\arrow["{q^*}", from=1-3, to=1-4]
	\arrow["{c^*}",from=1-2, to=1-3]
	\arrow[from=1-1, to=1-2]
\end{tikzcd}
\end{equation} 
from the cofiber sequence \begin{tikzcd}
	{\vee_{i\in I'}M^{(n-1)}_i} & {\#_{i\in I}M_i} & {M_k} & {\vee_{i\in I'}\Sigma M^{(n-1)}_i}
	\arrow["\iota", hook, from=1-1, to=1-2]
	\arrow["q", from=1-2, to=1-3]
    \arrow["c", from=1-3, to=1-4]
\end{tikzcd}, where $\iota$ is an inclusion map, $q$ is the quotient map, and $c$ is the connecting map in the cofiber sequence.
\end{proof}

%%%%%%%%%%%%%%%%%%%
%%%%%%%%%%%%%%%%%%%
In the subsequent theorem, we introduce the variables $\eta$, $\alpha$, $\beta$, and $\gamma$ as replacements for $\mu_0$, $\mu_1$, $\mu_2$, and $\mu_3$ respectively, from \cite[Theorem 2]{FujjiKO-grOfCP}. This substitution is intended to minimize confusion caused by excessive subscripts. Similarly, the symbols $\sigma$ and $\tau$ retain the same notation as in \cite[Theorem 2]{FujjiKO-grOfCP}. Thus, $\eta_i$, $\sigma_i$, and $\tau_i$ respectively correspond to $\eta$, $\sigma$, and $\tau$ for all $i\in I'$.

Also, to simplify the notation, we will use $f^*$ instead of $(\Sigma^sf)^*$, as the latter can be understood to be the maps $\iota$, $q$, or $c$ from the sequence \eqref{seq 2 (n-1 coppies to connected sum to single copy)}.

%%%%%%%%%%%%%%%%%%%

%%%%%%%%%%%%%%%%%%%
\begin{thm}\label{thm: ko group}
Let $k,n\geq 2$, and let $I'=\{1,2,\cdots,k-1\}$. Then
    \begin{enumerate}[(1)]
        \item \begin{enumerate}[(a)]

        \item $\wi{KO}^{0}(\#_k\Cp{4n})$ is a free module with a basis $$\{ q^*(\eta_k^j), q^*(\eta_k^{2n}), \eta_i^j~|~i\in I',~j=1,2,\cdots,2n-1 \}.$$

            \item $\wi{KO}^{0}(\#_k\Cp{4n+1})$ is a free module with a basis $$\{ q^*(\eta_k^j), q^*(\eta_k^{2n+1}), \eta_i^j~|~2q^*(\eta_k^{2n+1})=0,~i\in I',~j=1,2,\cdots,2n \}.$$

            \item $\wi{KO}^{0}(\#_k\Cp{4n+2})$ is a free module with a basis $$\{ q^*(\eta_k^j),  \eta_i^j~|~2\eta_i^{2n+1}=0,~i\in I',~j=1,2,\cdots,2n+1 \}.$$

            \item $\wi{KO}^{0}(\#_k\Cp{4n+3})$ is a free module with a basis $$\{ q^*(\eta_k^j),  \eta_i^j~|~i\in I',~j=1,2,\cdots,2n+1 \}.$$
        \end{enumerate}

        \item $\wi{KO}^{-1}(\#_k\Cp{n})=0$

    \item \begin{enumerate}
        \item $\wi{KO}^{-2}(\#_k\Cp{4n})$ is a free module with a basis $$\{q^*(\alpha\eta_k^{j}),q^*(\alpha\eta_k^{2n-1}) ,\alpha\eta_i^j , \sigma_i ~|~2\sigma_i=\alpha\eta_i^{2n-1},~i\in I',~j=0,1,\cdots,2n-2 \}.$$
        
       \item $\wi{KO}^{-2}(\#_k\Cp{4n+1})$ is a free module with a basis $$\{q^*(\alpha\eta_k^{j}),q^*(\alpha\eta_k^{2n}) ,\alpha\eta_i^j ~|~i\in I',~j=0,1,\cdots,2n-1 \}.$$

       \item $\wi{KO}^{-2}(\#_k\Cp{4n+2})$ is a free module with a basis $$\{q^*(\alpha\eta_k^{j}) ,\alpha\eta_i^j ~|~i\in I',~j=0,1,\cdots,2n \}.$$

       \item $\wi{KO}^{-2}(\#_k\Cp{4n+3})$ is a free module with a basis $$\{q^*(\alpha\eta_k^{j}),q^*(\sigma_k) ,\alpha\eta_i^j ~|~2\sigma_k=\alpha\eta_k^{2n+1},~i\in I',~j=0,1,\cdots,2n \}.$$
    \end{enumerate}    

    \item  $ \wi{KO}^{-3}(\#_k\Cp{n})=\begin{cases}
\Z_2^{k-1}                &; ~n=4m\\
\Z_2                 &;~n=4m+3~\\
0           &;  ~n=4m+1\text{ or }n=4m+2.
    \end{cases} $

    \item  \begin{enumerate}
        \item  $\wi{KO}^{-4}(\#_k\Cp{4n})$ is a free module with a basis
    $$\{ q^*(\beta\eta_k^j), \beta\eta_i^j~|~2\beta\eta_i^{2n-1}=0,~i\in I',~j=0,1,\cdots,2n-1\}.$$

     \item  $\wi{KO}^{-4}(\#_k\Cp{4n+1})$ is a free module with a basis
    $$\{ q^*(\beta\eta_k^j), \beta\eta_i^j~|~i\in I',~j=0,1,\cdots,2n-1\}.$$

     \item  $\wi{KO}^{-4}(\#_k\Cp{4n+2})$ is a free module with a basis
    $$\{ q^*(\beta\eta_k^j), q^*(\beta\eta_k^{2n}),\beta\eta_i^j~|~i\in I',~j=0,1,\cdots,2n-1\}.$$

    \item  $\wi{KO}^{-4}(\#_k\Cp{4n+3})$ is a free module with a basis
    $$\{ q^*(\beta\eta_k^j), q^*(\beta\eta_k^{2n+1}),\beta\eta_i^j~|~2\beta\eta_k^{2n+1}=0,~i\in I',~j=0,1,\cdots,2n\}.$$
    \end{enumerate}

    \item  $\wi{KO}^{-5}(\#_k\Cp{n})=0$.

   \item\begin{enumerate}
        \item $\wi{KO}^{-6}(\#_k\Cp{4n})$ is a free module with a basis
    $$\{  q^*(\gamma\eta_k^j),\gamma\eta_i^j~|~i\in I' \text{ and }~j=0,1,\cdots,2n-1\}.$$

    \item $\wi{KO}^{-6}(\#_k\Cp{4n+1})$ is a free module with a basis
    $$\{  q^*(\gamma\eta_k^j),q^*(\tau_k),\gamma\eta_i^j~|~2\tau_k=\gamma\eta_k^{2n},~i\in I' \text{ and }j=0,1,\cdots,2n-1\}.$$

    \item $\wi{KO}^{-6}(\#_k\Cp{4n+2})$ is a free module with a basis
    $$\{  q^*(\gamma\eta_k^j),q^*(\gamma\eta_k^{2n}),\gamma\eta_i^j,\tau_i|~2\tau_i=\gamma\eta_i^{2n},~i\in I' \text{ and }j=0,1,\cdots,2n-1\}.$$

   \item $\wi{KO}^{-6}(\#_k\Cp{4n+3})$ is a free module with a basis
    $$\{  q^*(\gamma\eta_k^j),q^*(\gamma\eta_k^{2n+1}),\gamma\eta_i^j|~i\in I' \text{ and }j=0,1,\cdots,2n\}.$$
    \end{enumerate}

    \item  $ \wi{KO}^{-7}(\#_k\Cp{n})=\begin{cases}
 \Z_2^{k-1}    &;~n=4m+2\\
 \Z_2         &;  ~n=4m+1\\
  0             &; ~n=4m\text{ or }n=4m+3.
    \end{cases} $

    \end{enumerate}
\end{thm}
% \tcor{We can write suspension of q in generating set to be more specific.}
\begin{proof}
The proof primarily relies on a specific instance of the sequence \eqref{seq 2 (n-1 coppies to connected sum to single copy)}, which can be stated as follows:
\begin{equation}\label{seq 2 for KO (n-1 coppies to connected sum to single copy)}
    \begin{tikzcd}[sep=small]
	\cdots {\underset{k-1}{\oplus}\wi{KO}^{-s-1}(\Cp{n-1})} & {\wi{KO}^{-s}(\Cp{n})} & {\wi{KO}^{-s}(\#_k\Cp{n})} & {\underset{k-1}{\oplus}\wi{KO}^{-s}(\Cp{n-1})\cdots}
	\arrow["{\iota^*}", from=1-3, to=1-4]
	\arrow["{q^*}", from=1-2, to=1-3]
	\arrow["{c^*}",from=1-1, to=1-2]
\end{tikzcd}
\end{equation}
We will establish that this sequence exhibits a splitting at  $\wi{KO}^{-s}(\#_k\Cp{n})$ for $k,n\geq 2$ and $0\leq s\leq7$. As a result, we obtain a basis for the group $\wi{KO}^{-s}(\#_k\Cp{n})$ in terms of the generators of  $\wi{KO}^{-s}(\Cp{n})$ and $\wi{KO}^{-s}(\Cp{n-1})$.  \\
The proof of \textit{(1)} will be presented following the computation of $\wi{KO}^{-7}(\#_k\Cp{n})$.\\
The proof of \textit{(2)} follows immediately by considering $s=1$ in the sequence \eqref{seq 2 for KO (n-1 coppies to connected sum to single copy)} using the fact that $\wi{KO}^{-1}(\Cp{n})=0$ for all $n\in\N$.\\
  % For \textit{(3)} consider $s=2$ in the sequence \ref{seq 2 for KO (n-1 coppies to connected sum to single copy)}. Then the map $c^*$ is always trivial. This is because the $\wi{KO}^{-3}(\Cp{n-1})$ is $0$ or $\Z_2$ and $\wi{KO}^{-2}(\Cp{n})$ is a free group for any $n\in\N$. In addition, the group $\wi{KO}^{-1}(\Cp{n})$ is trivial, implying the short exact sequence splits.
  To establish \textit{(3)}, we consider $s=2$ in the sequence \eqref{seq 2 for KO (n-1 coppies to connected sum to single copy)}. Remarkably, the map $c^*$ is always trivial in this case. This can be attributed to the fact that $\wi{KO}^{-3}(\Cp{n-1})$ is either $0$ or $\Z_2$, while $\wi{KO}^{-2}(\Cp{n})$ is a free group for any $n\in\N$. Moreover, the group $\wi{KO}^{-1}(\Cp{n})$ is trivial, leading to the conclusion that the short exact sequence splits.\\
Now, to prove \textit{(4)}, let us examine the case of $s=3$ in the sequence \eqref{seq 2 for KO (n-1 coppies to connected sum to single copy)}. The computation of $\wi{KO}^{-3}(\#_k\Cp{n})$ becomes straightforward when utilizing this sequence for $n=4m, 4m+1$, and $4m+2$. For $n=4m+3$, the sequence offers two options: either $0$ or $\Z_2$. To eliminate one of the choices, we can consider the sequence \eqref{seq 1 general short exact sequence} for $Y=BSO$ and $s=3$. By employing the identification $[\Sigma^sX,BSO]=\wi{KO}^{-s}(X)$, we obtain the following result:
  \begin{equation}\label{seq 1 for KO-groups}
    \begin{tikzcd}[sep=small]
	0 & {\faktor{\wi{KO}^{-s}(\Sp{2n)})}{\im{\Sigma^{s+1}h}^*}} & {\wi{KO}^{-s}(\#_k\Cp{n})} & {\underset{k-1}{\oplus}\wi{KO}^{-s}(\Cp{n-1})\opl{}\kr{\Sigma^s h}^*} & 0.
	\arrow[from=1-1, to=1-2]
	\arrow["{\wi{d^*}}", from=1-2, to=1-3]
	\arrow["{\iota^*}", from=1-3, to=1-4]
	\arrow[from=1-4, to=1-5]
\end{tikzcd}
\end{equation}
By considering $s=3$ and $n=4m+3$, and substituting the known value $ \wi{KO}^{-3}(\Cp{4m+2})=0$ into this sequence, we observe that $\wi{d^*}$ becomes an isomorphism. Furthermore, when we examine this sequence for $k=1$, we find that $\im{\Sigma^4h}^*$ is trivial. This completes the proof for this particular case.\\
The proofs of \textit{(5)(b)}, \textit{(5)(c)}, and \textit{(5)(d)} follow directly from applying the sequence \eqref{seq 2 for KO (n-1 coppies to connected sum to single copy)} with $s=4$. Regarding \textit{(5)(a)}, if we utilize the sequence \eqref{seq 2 for KO (n-1 coppies to connected sum to single copy)} for $\wi{KO}^{-4}(\#_k\Cp{4m})$, we can establish that the torsion part of $\wi{KO}^{-4}(\#_k\Cp{4m})$ is $\Z_2^t$, where the maximum possible value of $t$ in the exact sequence is $k-1$. Finally, using the sequence \eqref{seq 1 for KO-groups} for $s=4$ and $n=4m$ we establish that $t$ is precisely equal to $k-1$.\\
The proof of \textit{(6)} follows a similar approach to that of \textit{(2)}, as $\wi{KO}^{-5}(\Cp{n})$ is zero for all $n\in\N$.\\
Likewise, the proof of \textit{(7)} follows an analogous approach to that of \textit{(3)}.\\
To prove \textit{(8)}, consider $s=7$ in \ref{seq 2 for KO (n-1 coppies to connected sum to single copy)}. The cases $n=4m,4m+2$, and $4m+3$ follow immediately from the sequence. Further, the computation of the remaining case $n=4m+1$ follows analogously the way we computed $\wi{KO}^{-3}(\#_k\Cp{4m+3})$ using sequence  \ref{seq 1 for KO-groups}.\\
The proof of \textit{(1)(a),(b)} and \textit{(d)} follows immediately using sequence \eqref{seq 2 for KO (n-1 coppies to connected sum to single copy)} for $s=0$. If we consider $n=4m+2$ in \ref{seq 2 for KO (n-1 coppies to connected sum to single copy)} we get $\wi{KO}^0(\#_k\Cp{4t+2})=\Z^{2kt+1}\opl{}\Z_2^k ~\text{or}~\Z^{2kt+1}\opl{}\Z_2^{k-1}$. Finally, to eliminate one possibility, use the sequence \eqref{seq 1 for KO-groups} for $s=0$ and $n=4m+2$. 
This completes the proof for all cases.

\end{proof}

%%%%%%%%%%%%%%%%%%%%%%%%%%%%%%%%%%%%%%%%%%%%%%%%%%%%%%%%%%%%%%%%%%%%%%%%%%%%%%%%%%%%%%%%%%%%%%%%%%%%%%%%%%%%%%%%%%%%%%%%%%%%%%%%%%%%%%%%%%%%%%%%%%%%%%%%%%%%%%%%%%%%%%%%

 \section{Smooth normal invariant set}\label{sec: normal invariant set}

In \cite{Sullivan1996}, Sullivan established a significant bijection between the set of smooth normal invariants of $X$, denoted by $\mathcal{N}(X^n)$, and the group $[ X^n, F/O]$. The set $\mathcal{N}(X^n)$ represents different classes of equivalent smooth structures, where two structures are considered equivalent if they are cobordant.
Consequently, studying the group $[ X^n, F/O]$ becomes crucial for gaining a deeper understanding of surgery theory.
%Recall that, the smooth normal invariant set $\mathcal{N}(M)$ is in bijection with the group $[M,F/O]$. 
For $\#_k\Cp{n}$, we can examine the following long exact sequence obtained from the fiber sequence \eqref{f/o fiber seq}, to know $\mathcal{N}(\#_k\Cp{n})$;
\begin{equation}\label{f/o split exact seq}
    \begin{tikzcd}[sep=small]
	{\cdots[\#_k\Cp{n},SO]} & {[\#_k\Cp{n},SF]} & {[\#_k\Cp{n},F/O]} & {[\#_k\Cp{n},BSO]} & {[\#_k\Cp{n},BSF]\cdots}
	\arrow["{a_*}", from=1-1, to=1-2]
	\arrow["{b_*}", from=1-2, to=1-3]
	\arrow["{c_*}", from=1-3, to=1-4]
	\arrow["{f_*}", from=1-4, to=1-5]
\end{tikzcd}
\end{equation}

In this sequence, observe that, for all $k,n\in \N$ the group $[\#_k\Cp{n},SO]$ coincides with the $\wi{KO}^{-1}(\#_k\Cp{n})$, and which is trivial by \Cref{thm: ko group} \textit{(2)}. Therefore, in \eqref{f/o split exact seq}, the map $b_*$ is injective.
 To understand $c_*$, let us first shift our attention to compute the kernel of the map ${f_*:\wi{KO}^0(\#_k\Cp{n})\rightarrow\fbf{\#_k\Cp{n}}}$. Note that the case $k=1$ is already established in the \cite[pg 401]{brumfiel1968}. The result says that the kernel of the map $f_*:\wi{KO}^0(\Cp{n})\rightarrow\fbf{\Cp{n}}$ is torsion-free for all $n\in \N$, where $f:BSO\rightarrow BSF$ represents the canonical fibration. In the subsequent discussion, we aim to prove a similar result for the $k$-fold connected sum of $\Cp{n}$.

\begin{thm}\label{thm: ker of induced eat on ko group is torsion free}
	The kernel of $ f_*:\wi{KO}^0(\#_k\Cp{n})\ra\fbf{\#_k\Cp{n}} $ is torsion free for all $k, n\geq 2 $.
\end{thm}
\begin{proof} For the cases when $n=4m$ or $4m+3$, the result is straightforward since the groups $\wi{KO}^0(\#_k\Cp{n})$ are free.

In the remaining cases, we will demonstrate that the torsion part of $\wi{KO}^0(\#_k\Cp{n})$ maps injectively under $f_*$. To do so, we consider the following commutative diagram with the map $ J_* $ obtained from the $ J $-homomorphism;
\begin{equation}\label{commutative dig for ker of eta}
    \begin{tikzcd}
	\cdots0 & {\wi{KO}^0(\Sp{2n})} & {\wi{KO}^0(\#_k\Cp{n})} & {\opl{k}\wi{KO}^0(\Cp{n-1})} \\
	{\cdots\opl{k}[\Sigma\Cp{n-1},BSF]} & {[\Sp{2n},BSF]} & {[\#_k\Cp{n},BSF]} & {\opl{k}[\Cp{n-1},BSF]~,}
	\arrow[from=1-1, to=1-2]
	\arrow["{d^*}", from=1-2, to=1-3]
	\arrow["{\iota^*}", from=1-3, to=1-4]
	\arrow["{\opl{k}f_*}", from=1-4, to=2-4]
	\arrow["{f_*}"', from=1-3, to=2-3]
	\arrow["{\iota^*}"', from=2-3, to=2-4]
	\arrow[from=1-1, to=2-1]
	\arrow["{(\Sigma\xi)^*}"', from=2-1, to=2-2]
	\arrow["{J_*}"', from=1-2, to=2-2]
	\arrow["{d^*}"', from=2-2, to=2-3]
\end{tikzcd}
\end{equation}
 the horizontal rows are part of long exact sequences obtained by the cofiber sequence (4.1) from \cite{SB-RK-Priya2023smooth} for $ \#_{i\in I}M_i=\#_k\Cp{n} $, $ Y=BSO $ and $ BSF $ respectively, and the map $\xi=(\vee_kh_C)\circ p\,$ such that $ h $ is the Hopf map and $ p $ is the pinch map.

For the case when $n=4m+1$, we have $\widetilde{KO}^0(\#_k\Cp{4m+1})\cong\Z^{2km}\oplus\Z_2$ according to \Cref{thm: ko group} \textit{(b)}. Therefore, in order to demonstrate that $\kr{f_*}$ is torsion-free, it suffices to show that $\im{f_*(\Z_2)}$ is non-zero. 

We observe that $\wi{KO}^0(\Sp{8m+2})\cong\Z_2$, and the injectivity of the map $J_*$ \cite{Ravenel} implies that $\im{J_*}\cong\Z_2$. Thus, our objective is to show that the composition $d^*\circ J_*$ is non-zero, which is equivalent to proving that $\im{J_*}\nsubseteq \kr{d^*}=\im{(\Sigma\xi)^*}$. Note that, as indicated in the proof of the \cite[Theorem 4.2]{SB-RK-Priya2023smooth}, we have $\im{(\Sigma\xi)^*}=\im{(\Sigma h)^*}$. Hence, the problem reduces to verifying that $\im{J_*}\nsubseteq \im{(\Sigma h)^*}$.
%%%%%
	
	To prove this consider the following commutative diagram for an individual copy of $\Cp{n}$
	\[\begin{tikzcd}
	\cdots	0 & {\wi{KO}^0(\Sp{8m+2})=\Z_2} & {\wi{KO}^0(\Cp{4m+1})\cdots} \\
		{\cdots\fbf{\Sigma\Cp{4m}}} & {\fbf{\Sp{8m+2}}} & {\fbf{\Cp{4m+1}}\cdots~,}
		\arrow["{f_*}", from=1-3, to=2-3]
		\arrow["{J_*}", from=1-2, to=2-2]
		\arrow["{(\Sigma h)^*}"', from=2-1, to=2-2]
		\arrow["{d^*}"', from=2-2, to=2-3]
		\arrow["{d^*}", from=1-2, to=1-3]
		\arrow[from=1-1, to=2-1]
		\arrow[from=1-1, to=1-2]
	\end{tikzcd}\] 
	where $\wi{KO}^0(\Cp{4m+1})\cong\Z^{2m}\oplus\Z_2  $. According to \cite{brumfiel1968}, the kernel of $f_*:\wi{KO}^0(\Cp{4m+1})\to \fbf{\Cp{4m+1}}$ is torsion-free. Therefore, $f_*(\im{d^*})$ is non-zero since $d^*$ is injective. As a result, $d^*(\im{J_*}=\Z_2)$ is also non-zero. Consequently, we can conclude that $\im{J_*}\nsubseteq \kr{d^*} =\im{(\Sigma h)^*}$, which completes the proof in this case.
%%%%%%%%%%%%%%%%%%%%%%%%%%%%	
	
Lastly, consider the case $n=4m+2$ in \eqref{commutative dig for ker of eta}. Since $\wi{KO}^0(\Sp{8m+4})=\Z$ and by \Cref{thm: ko group}  \textit{(c)} the $\wi{KO}^0(\#_k\Cp{4m+2})=\Z^{2km+1}\oplus\Z_2^{k-1}$, we can observe that the map $\iota^*|_{\Z_2^{k-1}}$ is a monomorphism for the $\Z_2^{k-1}$ component. Additionally, since the kernel of $f_*:\wi{KO}^0(\Cp{4m+1})\to\fbf{\Cp{4m+1}}$ is a free group \cite{Ravenel}, the map $f_*|_{\Z_2}$ is also a monomorphism. This implies that the composition $(\oplus_kf_*) \circ \iota^*|_{\Z_2^{k-1}}$ is a monomorphism. Consequently, we can conclude that $f_*|_{\Z_2^{k-1}}:\wi{KO}^0(\#_k\Cp{4m+2})\to\fbf{\#_k\Cp{4m+2}}$ is a monomorphism, completing the proof.
	\end{proof}

%%%%%%%%%%%%%%%%%%%%%%%%%%%%%%%%%%%%%%%%%%%%

%%%%%%%%%%%%%%%%%%%%%%%%%%%%%%%%%%%%%%%%%%%%%%%%%%%%%%%%%%%%%%%%%%%%%%%%%%%%%%%%%%%%%%%%%%%%%%%%%%

Note that, from \cite{brumfiel1968}, it is known for all $n\in\N$, the groups $[\Cp{n}, SF]$ and $[\Cp{n}, BSF]$ are finite abelian groups. Thus, using the long exact sequence \eqref{seq 2 (n-1 coppies to connected sum to single copy)} for $Y=SF$ and $BSF$ respectively, we conclude that the groups $[\#_k\Cp{n},SF]$ and $[\#_k\Cp{n},BSF]$ are also finite abelian for all $k,n\in\N$.  
 Recall that, in the Theorem \ref{thm: ker of induced eat on ko group is torsion free}, we established that the kernel of the induced map ${f_*}$ is a free abelian group. Using the fact that the group $[\#_k\Cp{n},BSF]$ is finite abelian, we can conclude the following remark:

 \begin{rem}\label{rem: ker iso to free part}
    The kernel of  ${f_*:[\#_k\Cp{n},BSO]\rightarrow\fbf{\#_k\Cp{n}}}$ is isomorphic to the free part of $[\#_k\Cp{n},BSO]$, for all $k,n\in \N$.
 \end{rem}

%%%%%%%%%%%%%%%%%%%%%%%%%%%%%%%%%%%%%%%%%%%%%%%%%%%%%%%%%%%%%%%%%%%%%%%%%%%%%%%%%%%%%%%%%%%%%%
% \section{Conclusion \tcr{ change}}

Now, we have enough information to presents a result for the group $[\#_k\Cp{n}, F/O]$ with $k,n\geq 2$, which is $k$-fold connected sum analogue to the Brumfiles result from \cite[page 401]{Brumfiel1971HomotopyEOUnpublished}.
Note that, corresponding to the exact sequence \eqref{f/o split exact seq}, we have the following short exact sequence \begin{equation}
    \begin{tikzcd}\label{f/o short exact sequence}
	0 & {[\#_k\Cp{n},SF]} & {[\#_k\Cp{n},F/O]} & \kr{f_*} & 0 ,
	\arrow["b_*", from=1-2, to=1-3]
	\arrow[from=1-1, to=1-2]
	\arrow["c_*", from=1-3, to=1-4]
	\arrow[from=1-4, to=1-5]
\end{tikzcd}
\end{equation}
where $[\#_k\Cp{n},SF]$ is a finite abelian group and $\kr{f_*}$ is a free abelian group. Therefore, the short exact sequence \eqref{f/o fiber seq} must split. The formal statement of this finding, together with the result for $k=1$, is stated in the following theorem.
\begin{thm}\label{thm: f/o splitting}
     Let $k,n\in\N$. Then
        $$[\#_k\Cp{n},F/O]\cong\begin{cases}
            \pi_s^0(\#_k\Cp{n})\opl{}\Z^{k[\frac{n}{2}]} & \text{ ; if } n \text{ odd }\\
            \pi_s^0(\#_k\Cp{n})\opl{}\Z^{k[\frac{n-1}{2}]+1} & \text{ ; if } n \text{ even }.
         \end{cases} $$
\end{thm}
%%%%%%%%%%%%%%%%%%%%%%%%

Recall that, the set $\{ \eta^j~|~j=1,2,\cdots,n\}$ forms an integral basis for the free abelian group $\wi{K}^0(\Cp{n})$, for all $n\in \N$. Let $\ol{\eta}$ be the realification of $\eta$. Then we have the following result using \cite[\S4]{RobertLittle}.

\begin{thm}
   Let $\xi_1=24\ol{\eta}+98\ol{\eta}^2+111\ol{\eta}^3$, $\xi_2=240\ol{\eta}^2+380\ol{\eta}^3$, and $\xi_3=504\ol{\eta}^3$. Then the image of $c_*:[\#_k\Cp{n},F/O]\to [\#_k\Cp{n},BSO]$, for $n\leq 7$ and $k\in\Z$, is generated by $\xi_i$ and $q^*(\xi_i)$ for  $i=1,2,3$, where $q:\#_k\Cp{n}\lr \Cp{n}$ is the quotient map.
   
\end{thm}

% \tcbr{rough, re-frame the statement what it means}
% \begin{thm}
% \begin{enumerate}
%     \item If $u\in \Z^k\subseteq [\widehat{\#_k\Cp{4}},F/O]$ and $\xi_0(u)=m\xi_1$ with $m\in 14\Z$.
%     \item If $u\in \Z^{2k}\subseteq [\widehat{\#_k\Cp{5}},F/O]$ and $\xi_0(u)=m\xi_1+n\xi_2$ with $m\in 2\Z$.
%     \item If $u\in \Z^{2k}\subseteq [\widehat{\#_k\Cp{6}},F/O]$ and $\xi_0(u)=m\xi_1+n\xi_2$ with $m\in 992\Z$ and $n\in 31\Z$.
%      \item If $u\in \Z^{3k}\subseteq [\widehat{\#_k\Cp{7}},F/O]$ and $\xi_0(u)=m\xi_1+n\xi_2+l\xi_3$ with $m,n,l\in\Z$.
% \end{enumerate}
% \end{thm}

%%%%%%%%%%%%%%%%%%%%%%%%%%%%%%%%%%%%%%%%%%%%%%%%%%%%%%%%%%%%%%%%%%%%%%%%%%%%%%%%%%%%%%%%%%%%%%%%%%%%%%%%%%%%%%%%%%%%%%%%%%%%%%%%%%%%%%%%%%%%%%%%%%%%%%%%%%%%%%%%%%%%%%%%%%%%%%%%%%%%%%%
\section{Smooth tangential structure set}\label{smooth tangential structure set}

In the concluding section, we compute the smooth tangential structure set of $\#_k\Cp{n}$.
Recall from \cite[\S 6]{crowleyfinitegroupaction2015}
that, a \textit{smooth tangential structure set} of a closed smooth $n$-manifold $M$ with stable normal bundle $\nu_{M}$ of rank $k\gg n$, denoted by $\mathcal{S}^t_{\text{Diff}}(M)$, is
a triple $(N, f, b )$ such that $f : N \lr M$ is a homotopy equivalence and $b : \nu_N \lr \nu_M$ is a map of stable bundles. Two structures $(N_1 , f_1 , b_1 )$,
$(N_2 , f_2 , b_2 )$ are said to be equivalent if there is an $s$-cobordism $(U;N_1,N_2,F,B)$
where $F:U\to M$ is a simple homotopy equivalence, $F:\nu_U\to \nu_M$ is a bundle map and these data restrict to $(N_1 , f_1 , b_1 )$ and $(N_2 , f_2 , b_2 )$ at the boundary of $U$.

%-==============================================
In the smooth $PL$ category, there exists an analogous notion of the tangential $PL$ smoothings of a manifold $M$ with a tangent bundle $\tau_M$, denoted by $\mathcal{S}^t_{\text{PL}}(M)$ \cite[Page 194]{HomologySpheresasStationarysets}. It is a collection of triples $(V,t,\phi)$ with concordance relation, where $V$ is a smooth manifold, $t:V\to M$ is a piecewise differential homeomorphism, and $\phi:E(\tau_V\oplus k)\to E(\tau_M\oplus k)$ is a vector bundle isomorphism covering $t$.
%=================================================
%In the smooth $PL$ category, there exists an analogous notion of the $PL$-tangential structure set of a manifold $M$, denoted by $\mathcal{S}^t_{\text{PL}}(M)$, representing the set of equivalence classes of tangential $PL$-smoothings of $M$. 
Notably, \cite[Theorem 3.8]{HomologySpheresasStationarysets} establishes a one-to-one correspondence between $\mathcal{S}^t_{\text{PL}}(M)$ and the group $[M,PL]$, where $PL$ is colimit of groups $PL_n$, the self $PL$-homeomorphisms of $\R^n$ over $n$. Additionally, there exists a forgetful map $H:\mathcal{S}^t_{\text{PL}}(M)\to \mathcal{S}^t_{\text{Diff}}(M)$, which maps tangential $PL$ smoothing to tangential homotopy smoothing.

To compute these groups and understand the behavior of the map $H$ specifically for $\#_k\Cp{n}$, we explore some results.
For that, consider the following homotopy commutative diagram
\begin{equation}\label{dig: homo dig including pl f pl/o f/o bso bsf bpl etc}
  \begin{tikzcd}[row sep=1.5em]
	SO & SF & {F/O} & BSO & BSF \\
	SO & PL & {PL/O} & BSO & BPL \\
	& {\Omega(F/PL)}
	\arrow["f", from=1-4, to=1-5]
	\arrow["g"', from=2-4, to=2-5]
	\arrow[Rightarrow, no head, from=1-1, to=2-1]
	\arrow["l"', from=2-2, to=1-2]
	\arrow[from=2-1, to=2-2]
	\arrow["b", from=1-2, to=1-3]
	\arrow["v"', from=2-2, to=2-3]
	\arrow["c", from=1-3, to=1-4]
	\arrow["w"', from=2-3, to=2-4]
	\arrow[from=2-3, to=1-3]
	\arrow[Rightarrow, no head, from=2-4, to=1-4]
	\arrow[from=2-5, to=1-5]
	\arrow[from=3-2, to=2-2]
	\arrow[from=1-1, to=1-2]
\end{tikzcd}
\end{equation}
where the rows and columns are homotopy fiber sequences and all maps are natural maps. 
Here, $PL/O$ is a homotopy fiber of the map $BO\to BPL$, and $F/PL$ is a homotopy fiber of the map $BPL\to BSF$, where $BPL$ is the classifying space of $PL$.

%Then we have the following results.
%==============================================

% \begin{propn}\label{prop: conne sum cpn,F/PL computations}
% For all $k,n\in\N$
%     $$[\#_k\Cp{n},F/PL]\cong\begin{cases}
%     \Z^{k[\frac{n-1}{2}]+1}\oplus \Z_2^{k[\frac{n-1}{2}]} & ; \text{ if } n \text{ is even }\\
    
%     \Z^{k[\frac{n}{2}]}\oplus \Z_2^{k([\frac{n}{2}]-1)+1} & ; \text{ if } n \text{ is odd }
% \end{cases}$$
% \end{propn}
% \begin{proof}
% The computations of the group $[\#_k\Cp{n},F/PL]$ follow easily using the following short exact sequence, obtained using the sequence \cite[(4.1)]{SB-RK-Priya2023smooth},
%     \[\begin{tikzcd}
% 	0 & {P_{2n}} & {[\#_k\Cp{n},F/PL]} & {\oplus_k[\Cp{n-1},F/PL]} & 0~.
% 	\arrow[from=1-1, to=1-2]
% 	\arrow["{d^*}", from=1-2, to=1-3]
% 	\arrow["{\iota^*}", from=1-3, to=1-4]
% 	\arrow[from=1-4, to=1-5]
% \end{tikzcd}\]

% Here, due to Sullivan $[\Cp{m},F/PL]=\prod_{k=2}^m P_{2k}$ where $P_i=\Z,0,\Z_2,0$ for $i=0,1,2,3(mod ~4)$, respectively. 
 
% \end{proof}
% \tcbr{any remark}

% \begin{rem}
%     As $[\#_k\Cp{n},\Omega(F/PL)]=0$, 
% \end{rem}

Given a $PL$ self-homeomorphism $h$ of a smooth $n$-dimensional manifold $M$ and a homotopy of $h$ to the identity, there is an associated element of $[M,\Omega(F/PL)]$, which vanishes if the homotopy can be deformed to a pseudo-isotopy of $h$ to the identity. In particular, for $\#_k\Cp{n}$, we have the following result:
\begin{lemm}\label{lemm: loop F/pl=0}
Every $PL$ self-homeomorphism of $\#_k\Cp{n}$ homotopic to identity is pseudo-isotopic to the identity.
\end{lemm}
\begin{proof}
    The  \cite[Corollary 2.3]{brumfiel1968} proves that $[\Cp{n},\Omega(F/PL)]=0$ for all $n\in \N$. Thus, using \Cref{lemm: skeleton sandwitch sesquence}, we get that $[\#_k\Cp{n},\Omega(F/PL)]=0$ for all $k,n\in \N$, implying the statement.
\end{proof}

\begin{thm}\label{thm: stpl to stdiff inje}
    The forgetful map $H:\mathcal{S}^t_{\text{PL}}(\#_k\Cp{n})\to \mathcal{S}^t_{\text{Diff}}(\#_k\Cp{n})$ is injective, for all $k,n\in\N$.
\end{thm}
\begin{proof}
   The statement follows immediately by using \Cref{lemm: loop F/pl=0} in the commutative diagram mentioned in \cite[Theorem 3.9]{HomologySpheresasStationarysets} for $M=\#_k\Cp{n}$.
\end{proof}

In the category of $PL$-manifolds, the notion $PL$-normal invariant set $\mathcal{N}_{\text{PL}}(M)$ which is in bijection with $[M,F/PL]$, is analogous to smooth normal invariant set. For the next proposition, we compute the group $[\#_k\Cp{n},F/PL]$, using the sequence \cite[(4.1)]{SB-RK-Priya2023smooth}, and the fact that $[\Cp{m},F/PL]=\prod_{k=2}^m P_{2k}$ where $P_i=\Z,0,\Z_2,0$ for $i=0,1,2,3(mod ~4)$, respectively,  
\begin{equation}\label{eq: conn cp to f/pl}
    [\#_k\Cp{n},F/PL]\cong\begin{cases}
    \Z^{k[\frac{n-1}{2}]+1}\oplus \Z_2^{k[\frac{n-1}{2}]} &  \text{ if } n \text{ is even }\\
    
    \Z^{k[\frac{n}{2}]}\oplus \Z_2^{k([\frac{n}{2}]-1)+1} &  \text{ if } n \text{ is odd }.
\end{cases}
\end{equation}
%===============================================
\begin{thm}\label{thm: connetced sum Cpn PL isomorphic to PL/o} For all $k,n\in\N$ we have following results.
\begin{enumerate}
     \item  The map $l_*: [\#_k\Cp{n},PL]\lr [\#_k\Cp{n},SF]$ is injective and is an isomorphism on odd torsion.

    \item  The map $w_*: [\#_k\Cp{n},PL/O]\lr [\#_k\Cp{n},BSO]$ is zero.

    \item   The natural induced map $f_*:[\#_k\Cp{n},PL]\lr [\#_k\Cp{n},PL/O]$ is an isomorphism. 

\end{enumerate}    
\end{thm}

\begin{proof}
\begin{enumerate}
    \item Using \Cref{lemm: loop F/pl=0}
   we get that $[\#_k\Cp{n},\Omega(F/PL)]=0$. Therefore, from the diagram \ref{dig: homo dig including pl f pl/o f/o bso bsf bpl etc} the induced map $l_*:[\#_k\Cp{n},PL]\lr [\#_k\Cp{n},SF]$ is injective for all $k,n\in \N$. 

    Now, from \eqref{eq: conn cp to f/pl} it is clear that $[\#_k\Cp{n},F/PL]$ has no odd torsion subgroup. This makes $l_*$  an isomorphism after localizing at odd primes.

\item  It is known from \cite{KirbySiebenmannFoundationalEssays} that, the group $[\#_k\Cp{n},PL/O]$ is finite abelian for all $k,n\in \N$. Now, consider the following homotopy commutative diagram obtained using \eqref{dig: homo dig including pl f pl/o f/o bso bsf bpl etc}
\[\begin{tikzcd}
	{[\#_k\Cp{n},BSO]} & {[\#_k\Cp{n},BPL]} \\
	{[\#_k\Cp{n},BSO]} & {[\#_k\Cp{n},BSF]}
	\arrow["{g_*}", from=1-1, to=1-2]
	\arrow[Rightarrow, no head, from=1-1, to=2-1]
	\arrow["{f_*}"', from=2-1, to=2-2]
	\arrow[from=1-2, to=2-2]
\end{tikzcd}\]
In \Cref{thm: ker of induced eat on ko group is torsion free}, we proved that the torsion part of $[\#_k\Cp{n},BSO]$ gets mapped injectively under $f_*$, and so is under $g_*$.
Hence, in \eqref{dig: homo dig including pl f pl/o f/o bso bsf bpl etc} the $\im{w_*}$ is a free abelian group. Since $[\#_k\Cp{n},PL/O]$ is a finite abelian group, the map $w_*$ must be zero.

\item The \Cref{thm: ko group}$(2)$ gives $[\#_k\Cp{n},SO]=0$ for all $k,n\in\N$. Thus, the map $f_*:[\#_k\Cp{n},PL]\lr [\#_k\Cp{n},PL/O]$ is injective, and the surjectivity follows from $(2)$ of this theorem and by considering diagram \eqref{dig: homo dig including pl f pl/o f/o bso bsf bpl etc}.
\end{enumerate} 
\end{proof}

 \begin{rem} 
From \Cref{thm: connetced sum Cpn PL isomorphic to PL/o} $(3)$, the concordance structure set $\cc{\#_k\Cp{n}}$ (which is in bijection with $[\#_k\Cp{n},PL/O]$, see \cite{Cairns-Hirsch-Mazur}) and the $PL$-tangential structure set $\mathcal{S}^t_{\text{PL}}(\#_k\Cp{n})$ are isomorphic, for all $k,n\in\N$.
 \end{rem}
%------------------------------------------------
Now, we use \Cref{thm: connetced sum Cpn PL isomorphic to PL/o}, and compute $\pi_0^s(\#_k\Cp{5})$.
\begin{thm}\label{thm: cp5 sf}
    $\pii{s}^0(\#_k\Cp{5}) \cong\Z_2^{2k}\opl{}\Z_3$
\end{thm}
\begin{proof} From \Cref{co-homotopy groups of cpn} $(iii)$, we obtained $[\#_k\Cp{5},SF]\cong \Z_2^{2k}\oplus\Z_3$ or $\Z_2^{2k-2}\oplus\Z_4\oplus\Z_3$ depending upon the splitting or non-splitting of the corresponding short exact sequence. To conclude this, let us consider the following commutative diagram
\[\begin{tikzcd}
	0 & {[\#_k\Cp{n},PL/O]} & {[\#_k\Cp{n},F/O]} & {[\#_k\Cp{n},F/PL]} & \dots \\
	0 & {[\#_k\Cp{n},SF]} & {[\#_k\Cp{n},F/O]} & {[\#_k\Cp{n},BSO]} & \dots
	\arrow[from=1-1, to=1-2]
	\arrow["{f'_*}", from=1-4, to=1-5]
	\arrow["{b'_*}", from=1-2, to=1-3]
	\arrow["{c'_*}", from=1-3, to=1-4]
	\arrow[from=2-1, to=2-2]
	\arrow[ from=1-2, to=2-2]
	\arrow["{b_*}"', from=2-2, to=2-3]
	\arrow[Rightarrow, no head, from=1-3, to=2-3]
	\arrow["{c_*}"', from=2-3, to=2-4]
	\arrow[from=1-4, to=2-4]
	\arrow["{f_*}"', from=2-4, to=2-5]
\end{tikzcd}\]

    The map $[\#_k\Cp{n},PL/O]\lr [\#_k\Cp{n},SF]$ is injective because of  \Cref{thm: connetced sum Cpn PL isomorphic to PL/o}$(1)$ and $(3)$, which maps a homeomorphism to a tangential homotopy equivalence. The proof of \Cref{thm: f/o splitting} uses the following short exact sequence
\[\begin{tikzcd}
	0 & {[\#_k\Cp{n},SF]} & {[\#_k\Cp{n},F/O]} & {\im{c_*}} & 0
	\arrow[from=1-1, to=1-2]
	\arrow["{b_*}", from=1-2, to=1-3]
	\arrow["{c_*}", from=1-3, to=1-4]
 \arrow[from=1-4, to=1-5]
\end{tikzcd}\]
    which is a split short exact sequence. This shows that the above short exact sequence
    \[\begin{tikzcd}
	0 & {[\#_k\Cp{n},PL/O]} & {[\#_k\Cp{n},F/O]} & {\im{c'_*}} & 0
	\arrow[from=1-1, to=1-2]
	\arrow["{b'_*}", from=1-2, to=1-3]
	\arrow["{c'_*}", from=1-3, to=1-4]
 \arrow[from=1-4, to=1-5]
\end{tikzcd}\]
is also a split short exact sequence.

Furthermore, according to \cite[Theorem 5.8]{SB-RK-Priya2023smooth} $(i)$, we have $[\#_k\Cp{5},PL/O]\cong\Z_2^{k+1}\oplus\Z_3$, and from \eqref{eq: conn cp to f/pl}, we find that $[\#_k\Cp{5},F/PL]\cong\Z^{2k}\oplus \Z_2^{k+1}$. Consequently, we conclude that the torsion part of $[\#_k\Cp{n},F/O]$ does not contain $\Z_4$. Hence, $[\#_k\Cp{n},SF]$ does not contain $\Z_4$, thereby completing the proof.\end{proof}
%%%%%%%%%%%%%%%%%%%%%%%%%%%%%%%%%%%%%%%%%%

Recall that, the smooth tangential structure set of a simply-connected manifold $M$ is a part of the following smooth tangential surgery exact sequence for $M$ 
\begin{equation}\label{short exact seq for Stdiff}
    \begin{tikzcd}
	{L_{2n+1}(\Z \bpi))} & {\mathcal{S}^{t}_{\text{Diff}}(M)} & {\mathcal{N}^t_{\text{Diff}}(M)} & {L_{2n}(\Z\bpi)},
	\arrow[from=1-1, to=1-2]
	\arrow["\eta",  from=1-2, to=1-3]
	\arrow["s", from=1-3, to=1-4]
\end{tikzcd}
\end{equation}
where $\mathcal{N}^t_{\text{Diff}}(M)$ is the set of tangential normal invariants of $M$, and $L_i(\Z\bpi)=\Z,0,\Z_2,0$ for $i\equiv 0,1,2,3 (mod ~4)$ respectively, the map $s$ is the surgery obstruction map which is a homomorphism if $n$ is odd, else is not a homomorphism.
Here, note that, there is a bijection $$\mathcal{N}^t_{\text{Diff}}(M)\cong [M,SF].$$

We have computed $\mathcal{N}^t_{\text{Diff}}(\#_k\Cp{n})$ as detailed in \Cref{co-homotopy groups of cpn,pi0 of cp4 and 7,thm: cp5 sf}. Since $L_{11}=0$, the map $\eta:\mathcal{S}^{t}_{\text{Diff}}(\#_k\Cp{n})\to \mathcal{N}^t{\text{Diff}}(\#_k\Cp{n}) $ is injective for all $k,n\in\N$. More precisely, we have the following result:
\begin{thm}\label{thm: stdiff}
Let $k,n\in\N$. Then
\begin{enumerate}
    \item The map $\eta:\mathcal{S}^{t}_{\text{Diff}}(\#_k\Cp{n})\to \mathcal{N}^t_{\text{Diff}}(\#_k\Cp{n}) $ is an isomorphism for $n=3,4$ and $6$.

    \item The map $\eta:\mathcal{S}^{t}_{\text{Diff}}(\#_k\Cp{5})\to \mathcal{N}^t_{\text{Diff}}(\#_k\Cp{5}) $ has image isomorphic to ${(\sfrac{\Z}{2})^{2k-1}}$.

    \item The map $\eta:\mathcal{S}^{t}_{\text{Diff}}(\#_k\Cp{7})\to \mathcal{N}^t_{\text{Diff}}(\#_k\Cp{7}) $ has image isomorphic to $\mathcal{S}^t_{\text{PL}}({\#_k\Cp{7}})$.

% %-------------------------------------
%     \item $\mathcal{S}^{t}_\text{Diff}(\#_k\Cp{3})\cong\pi^0_s(\#_k\Cp{3})\cong\Z_2$

%     \item $\mathcal{S}^{t}_\text{Diff}(\#_k\Cp{4})\cong \pi^0_s(\#_k\Cp{4})\cong\Z_2^{k+1}$

%     \item $\mathcal{S}^{t}_\text{Diff}(\#_k\Cp{6})\cong \pi^0_s(\#_k\Cp{6})\cong\Z_2^{2k-1}\opl{}\Z_3^k $

%     \item $\mathcal{S}^{t}_\text{Diff}(\#_k\Cp{7})\cong [\#_k\Cp{7},PL/O]\cong\Z_2^{k+1}\opl{}\Z_3^{k-1}$
% %-------------------------------------

\end{enumerate}
\end{thm}
\begin{proof}

Consider the commutative diagram mentioned in \cite[Theorem 3.9]{HomologySpheresasStationarysets} for $M=\#_k\Cp{n}$. Then using Theorems \ref{lemm: loop F/pl=0} and \ref{thm: connetced sum Cpn PL isomorphic to PL/o}$(3)$ together, we obtain that $[\#_k\Cp{n},PL/O]$ sits injectively in $\mathcal{S}^t_\text{Diff}(\#_k\Cp{n})$.
In addition to this, Proposition \ref{thm: connetced sum Cpn PL isomorphic to PL/o}$(1)$ gives, $[\#_k\Cp{n},PL/O]$ is injective in $\pi^0_s(\#_k\Cp{n})$.

\begin{enumerate}
    \item For $n=6$, the isomorphism between the groups $\pi^0_s(\#_k\Cp{6})$ and $[\#_k\Cp{6},PL/O]$ \cite[Corollary 3.6(iii)]{SB-RK-Priya2023smooth} confirms statement \textit{(1)} for these cases.

For the remaining case, consider the following commutative diagram
\begin{equation}\label{dig: surgery exact seq for cpn}
    \begin{tikzcd}
	0 & {\mathcal{S}^{t}_{\text{Diff}}(\#_k\Cp{n})} & {[\#_k\Cp{n},SF]} & {L_{2n}} \\
	&& {\oplus_k[\Cp{n},SF]} & {L_{2n}}
	\arrow[from=1-1, to=1-2]
	\arrow[from=1-2, to=1-3]
	\arrow["s", from=1-3, to=1-4]
	\arrow["s", from=2-3, to=2-4]
	\arrow["{q^*}"', from=2-3, to=1-3]
	\arrow[Rightarrow, no head, from=2-4, to=1-4]
\end{tikzcd}
\end{equation}

In case of $n=3$ and $4$, the surgery obstruction map $s:[\Cp{n},SF]\lr L_{2n}$ is zero, so is the $s:\opl{k}[\Cp{n},SF]\lr L_{2n}$. Also, note that, the map $q^*:\opl{k}[\Cp{n},SF]\lr [\#_k\Cp{n},SF]$ is surjective for these case (see the proof of \Cref{pi0 of cp4 and 7} $(i)$ for n=4). Therefore, the map $s:[\#_k\Cp{3},SF]\lr L_{2n}$ is zero, making the statement $1$ clear.

\item For $n=5$, the following commutative square
\[\begin{tikzcd}
	{[\Sp{10},F/PL]} & {[\#_k\Cp{5},F/PL]} \\
	{L_{10}} & {L_{10}}
	\arrow["\sigma", from=1-2, to=2-2]
	\arrow["{d^*}", from=1-1, to=1-2]
	\arrow["\cong"', from=1-1, to=2-1]
	\arrow[Rightarrow, no head, from=2-1, to=2-2]
\end{tikzcd}\]
%Using sequence \cite{??} for $M=\#_k\Cp{5}$ and $Y=F/PL$, together with \Cref{lemm: loop F/pl=0}, we obtain that $d^*:[\Sp{10},F/PL]\to[\#_k\Cp{5},F/PL]$ is injective. 
gives that the map $\sigma$ is non-zero. Therefore, the statement $(2)$ follows immediately by applying this in the commutative diagram referenced in \cite[Theorem 3.9]{HomologySpheresasStationarysets} for $M=\#_k\Cp{5}$, as the map  $s:[\#_k\Cp{5},SF]\to L_{10}$ becomes non-zero.

\item For $n=7$, the $s:[\Cp{7},SF]\lr L_{14}$ is a non-zero homomorphism. From the proof of \Cref{pi0 of cp4 and 7} $(ii)$, note that $q^*:\opl{k}[\Cp{7},SF]\lr [\#_k\Cp{7},SF]$ is surjective. Hence in \eqref{dig: surgery exact seq for cpn} the map $s:[\#_k\Cp{7},SF]\lr L_{14}$ is non-zero homomorphism. 
\end{enumerate}
This completes the proof in all cases.    
\end{proof}

%================================================

\vspace{1em}

%%%%%%%%%%%%%%%%%%%%%%%%%%%%%%%%%%%%%%%%%%%%%%%%%%%%%%%%%%%%%%%%%%%%%%%%%%%%%%%%%%%%%%%%%%%%%%%%%%%%%%%%%%%%%%%%%%%%%%%%%%%%%%%%%%%%%%%%%%%%%%%%
% \tcc{Review: No applications are given, and no structural conclusions
% are drawn. Calculations
% of this type would best be published along with any geometric applications}

% There is a crucial relationship between the groups discussed above and tangential surgery exact sequence \cite[Theorem 3.9]{HomologySpheresasStationarysets}.
%%%%%%%%%%%%%%%%%%%%%%%%%%%%%%%%%%%%%%%%%%%%%%%%

\noindent Acknowledgement: 
The author extends sincere gratitude to R. Kasilingam for providing valuable guidance throughout the development of this paper. His expertise and advice have significantly enriched the content. The author also extends appreciation to S. Gurjar for general comments and support. Furthermore, the author thanks the University Grants Commission for their support in the form of the Senior Research Fellowship.

%%%%%%%%%%%%%%%%%%%%%%%%%%%%%%%%%%%%%%%%%%%%%%%%%%%%%%%%%%%%%%%%%%%%%%%%%%%%%%%%%%%%%%%%%%
\providecommand{\bysame}{\leavevmode\hbox to3em{\hrulefill}\thinspace}
\providecommand{\MR}{\relax\ifhmode\unskip\space\fi MR }
\providecommand{\MRhref}[2]{%
  \href{http://www.ams.org/mathscinet-getitem?mr=#1}{#2}
}
\providecommand{\href}[2]{#2}

%%%%%%%%%%%%%%%%%%%%%%%%%%%%%%%%%%%%%%%%%%%%%%%%%%%%%%%%%%%%%%%%%%%%%%%%%%%%%%%%%%%%%%%%%%%%%%%%%%%%%%%%%%%%%%%%%%%%%%%%%%%%%%%%%%%%%%%%%%%%%%%%%%%%%%%%%%%%%%%%%%%%%%%%%%%%%%%%%%%%%%%%%%%%%%%%%%%%%%%%%%%%%%%%%%%%%%%%%%%%%%%%%%%%%%%%%%%%%%%
\end{document}